\documentclass[a4paper]{amsart}

\usepackage[all]{xy}
\usepackage{amsfonts,amssymb}

\usepackage[hidelinks]{hyperref}

\theoremstyle{plain}
\newtheorem{theorem}{Theorem}[section]
\newtheorem{proposition}[theorem]{Proposition}
\newtheorem{lemma}[theorem]{Lemma}

\newtheorem{atheorem}{Theorem}

\newtheorem{acorollary}[atheorem]{Corollary}

\newtheorem*{question*}{Question}

\theoremstyle{definition}
\newtheorem{definition}[theorem]{Definition}

\theoremstyle{remark}
\newtheorem{remark}[theorem]{Remark}

\newcommand{\dd}{\partial}
\newcommand{\ddual}{\partial^\vee}

\def\gauss{\mathrm{gauss}}
\def\degree{d}
\def\UUU{U_{\degree,n}}

\def\P{\mathbb{P}} 
\renewcommand{\SS}[1]{\tilde{\mathcal{P}}^{\perp}_{#1}} 
\newcommand{\PP}[1]{\mathcal{P}_{#1}} 
\newcommand{\SSH}[1]{\widetilde{\mathrm{Gr}}_{#1}} 
\def\CS{\mathcal{A}_{n-1}(\CP^n)} 
\def\Cbun{\PP{1}(T\CP^n)} 
\def\PCbun{\mathbb{P}(T\CP^n\oplus \epsilon)} 

\def\GCbun{\Gamma(\Cbun)} 
\newcommand{\scan}{\mathfrak{s}}

\def\U{U}
\def\PU{PU}
\def\BU{BU}
\def\EU{EU}

\def\Vect{\mathbf{Vect}}
\def\Top{\mathbf{Top}}

\def\map{\mathrm{map}}

\def\cA{\mathcal{A}}

\def\lra{\longrightarrow}
\def\bZ{\mathbb{Z}}
\def\bQ{\mathbb{Q}}
\def\bC{\mathbb{C}}
\def\bR{\mathbb{R}}
\def\CP{\mathbb{C}P}
\def\Gr{\mathrm{Gr}}
\def\Diff{\mathrm{Diff}}

\def\Emb{\mathrm{Emb}}
\def\ACEmb{\widetilde{\mathrm{Emb}}}

\def\Id{\mathrm{Id}}

\usepackage[x11names]{xcolor}
\definecolor{darkred}{RGB}{139,0,0}
\definecolor{darkblue}{RGB}{0,0,139}
\definecolor{darkgreen}{RGB}{0,100,0}
\definecolor{darkmagenta}{RGB}{139,0,120}

\title[The scanning map and the space of smooth hypersurfaces]{On the scanning map and the space of smooth complex projective hypersurfaces}
\author{Ángel Javier Alonso}
\author{Federico Cantero-Morán}
\thanks{F. Cantero-Morán was funded by grants SI3/PJI/2021-00505 from Comunidad de Madrid and PID2019-108936GB-C21 from
the Spanish government. Á.\ J.\ Alonso acknowledges the support of the Austrian Science Fund (FWF) grant P 33765-N}
\begin{document}

\begin{abstract}
The space of degree $\degree$ smooth projective hypersurfaces of $\CP^n$ admits a scanning map to a certain space of sections. We compute a rational homotopy model of the action by conjugation of the group $U(n+1)$ on this space of sections, from which we deduce that the scanning map induces a monomorphism on cohomology when $\degree > 2$. Our main technique is the rational homotopy theory of Sullivan and, more specifically, a Sullivan model for the action by conjugation of a connected topological group on a space of sections.
\end{abstract}
\maketitle

\section{Introduction}
The projectivization $\PCbun$ of the Whitney sum of the tangent space of $\CP^n$ with a trivial line bundle is a $\CP^n$-bundle over $\CP^n$. Each section of this bundle has an associated integer (its \emph{degree}), and we denote by $\Gamma(\PCbun)_\degree$ the component of those sections of degree $\degree$. The space $\UUU$ of smooth hypersurfaces of $\CP^n$ of degree $\degree$ admits a continuous map (called \emph{scanning map})
\begin{equation*}
	\scan \colon \UUU\lra \Gamma(\PCbun)_\degree
\end{equation*}
that can be used to produce classes in the cohomology ring of $\UUU$. Recently Aumonier has proven that this map induces an isomorphism on integral homology in degrees $*\leq \frac{d-3}{2}$ \cite{Aumonier21, Aumonier23}. In this note we investigate the behaviour of the scanning map outside this range of degrees, obtaining the following 
\begin{atheorem}\label{thm:a}
	The scanning map induces a monomorphism on rational cohomology if $\degree \geq 3$. 
\end{atheorem}

\subsection{The orbit map} The group $\PU(n+1)$ acts on the right on the space $\UUU$ and by conjugation on the space of sections, and the scanning map is equivariant with respect to this action. The fibre bundle $\PCbun\to \CP^n$ is trivial, thus from \cite{Moller-Raussen} we obtain that the space of sections of degree $\degree\neq 1$ has the rational homotopy type of $\PU(n+1)$, while the space of sections of degree $\degree = 1$ has the rational homotopy type of $\CP^n\times \U(n)$. We will prove the following theorem, from which Theorem \ref{thm:a} follows:
\begin{atheorem}\label{thm:main}
     Let $f$ be a section in $\Gamma(\PCbun)_\degree$. The orbit map
\begin{align*}
	\{f\}\times \PU(n+1)&\lra 	\Gamma(\PCbun)_\degree
\end{align*}
induces an isomorphism in rational cohomology if $\degree\neq 0,1,2$. 
\end{atheorem}
If $\degree = 0,1$ the induced homomorphism on rational cohomology is trivial while if $\degree = 2$, its image is the subalgebra generated by the looping of the odd Chern classes. 

This orbit map is one of the main traditional tools to understand the cohomology of the space of smooth hypersurfaces. From Theorem \ref{thm:main} we can recover the following theorem of Peters and Steenbrink:
\begin{acorollary}\cite{PS}
	The orbit map $\PU(n+1)\to \UUU$ induces an epimorphism on rational cohomology if $\degree \geq 3$.
\end{acorollary}
Additionally, combining Aumonier's theorem with Theorem \ref{thm:main} we recover the following theorem of Tommasi (with a slightly worse stable range):
\begin{acorollary}\cite{Tommasi}
	The orbit map $\PU(n+1)\to \UUU$ induces an isomorphism on rational cohomology in degrees $*< \frac{d+1}{2}$.
\end{acorollary}

In light of Aumonier's theorem,  the following question prompts naturally.
\begin{question*}
	Does Theorem \ref{thm:a} hold with integral coefficients?     
\end{question*}

\subsection{Almost-complex hypersurfaces} In the appendix we will define an almost-complex hypersurface of $\CP^n$ as a compact oriented smooth submanifold $W$ of $\CP^n$ of codimension $2$, together with a deformation of its tangent bundle $TW\subset T\CP^n|_{W}$ to a complex subbundle of $T\CP^n|_{W}$. The \emph{degree} of $W$ is the intersection product of the image of the fundamental class of $W$ in $H_{2n-2}(\CP^n)$ with the fundamental class of $\CP^1\subset \CP^n$. The set of all almost-complex hypersurfaces of $\CP^n$ forms a topological space denoted $\CS$ where hypersurfaces of different degree live in different components. Let $\CS_\degree$ denote the subspace of those of degree $\degree$. The scanning map $\scan$ factors equivariantly as
\[
	\UUU\lra \CS_{\degree}\lra \Gamma(\PCbun)_\degree
\]
and from our Theorem \ref{thm:main} we deduce, as before, the following corollaries:
\begin{acorollary}
	The map $\CS_{\degree}\lra \Gamma(\PCbun)_\degree$ induces a monomorphism on rational cohomology if $\degree \geq 3$.
\end{acorollary}
\begin{acorollary}
	If $W$ is an almost-complex hypersurface of degree $\degree$, then the orbit map $\PU(n+1)\times \{W\}\lra \CS_{\degree}$ induces an epimorphism on rational cohomology if $\degree \geq 3$.
\end{acorollary}

From our Theorem \ref{thm:main} and either Tommasi's or Aumonier's theorem, we deduce that if $*< \frac{d-3}{2}$, then $H^*(\CS_{\degree};\bQ)$ contains $H^*(\UUU;\bQ)$ as a split summand.
\begin{question*}
	What is the relative cohomology of the inclusion $\UUU\subset \CS_\degree?$
\end{question*}
As far as we know, the concept of almost-complex submanifold has not appeared before in the literature. 

\subsection{Techniques} The computation of the rational homotopy type of these section spaces will be performed using the Sullivan model of a space of sections of Brown and Szczarba \cite{BS:models}. The main tool to calculate the orbit map in rational cohomology is Theorem \ref{thm:newBS}, where we give a rational homotopy model of the action by conjugation of a connected topological group $G$ on a space of sections. Nils Prigge has informed us that a similar model had been considered by Haefliger in \cite{Haefliger-GF}.


The scanning map $\scan$ is similar to the one used by Madsen and Weiss in their solution of the Mumford conjecture on the cohomology of the moduli space of curves \cite{MW}. 

\subsection*{Acknowledgements} The authors thank Alexis Aumonier for some fruitful conversations, Oscar Randal-Williams for bringing their attention to Tommasi's paper and Aniceto Murillo for some comments on the rational homotopy techniques used in this paper. Further thanks go to Nils Prigge for pointing us to the reference \cite{Haefliger-GF}. They strongly thank the referee for his thorough work.

\section{Rational models of section spaces}
In this section, we first recall the model of Brown and Szczarba~\cite{BS:models} of the rational homotopy type of a space of sections. Afterwards, we give a model of the action by conjugation of a group on the space of sections. We assume familiarity with rational homotopy theory and Sullivan models \cite{FHT, FOT, Sullivan}.

 A \emph{Sullivan model of a path-connected space} $X$ is a Sullivan algebra $(\Lambda V,\dd)$ together with a quasi-isomorphism $(\Lambda V,\dd)\to A_{PL}(X)$. A \emph{strong rational model of a path-connected space} $X$ is a commutative differential graded algebra $(A,\dd)$ together with a quasi-isomorphism $(A,\dd)\to A_{PL}(X)$. A \emph{rational model of a path-connected space} $X$ is a commutative differential graded algebra $(A,\dd)$ together with a quasi-isomorphism $(\Lambda V,\dd)\to (A,\dd)$, where $(\Lambda V,\dd)$ is a Sullivan model of $X$.

Let $f\colon X\to Y$ be a map between topological spaces and let $(\Lambda V,\dd)$ be a Sullivan model of $Y$ and $(\Lambda V',\dd')$ a Sullivan model of $X$. A \emph{Sullivan model} of $f$ is a map $\varphi\colon (\Lambda V,\dd)\to (\Lambda V',\dd')$ and the following diagram commutes:
 \[
    \xymatrix{
        (\Lambda V,\dd)\ar[r]^{\simeq}\ar[d]^{\varphi} & A_{PL}(Y)\ar[d]^{A_{PL}(f)} \\
        (\Lambda V',\dd')\ar[r]^{\simeq} & A_{PL}(X).
    }
 \]
 Let $(A,\dd)$ and $(A',\dd')$ be rational models of $X$ and $Y$. A \emph{rational model} of $f$ is a map $\phi\colon (A,\dd)\to (A',\dd')$, such that the following diagram commutes:
 \[
    \xymatrix{
        (\Lambda V,\dd)\ar[r]^{\simeq}\ar[d]^{\varphi} & (A,\dd)\ar[d]^{\phi} \\
        (\Lambda V',\dd')\ar[r]^{\simeq} & (A',\dd').
    }
 \]
 Sullivan algebras are cofibrant in the category of cdgas. As a consequence, if $X,Y$ are finite type rational spaces, with the previous notation there is a bijection between homotopy classes of maps from $(\Lambda V',\dd')$ to $(A,\dd)$ and homotopy classes of maps from $X$ to $Y$. They also have the left lifting property with respect to surjections of commutative differential graded algebras.
 
A \emph{relative Sullivan model of a fibration} $F\to E\to X$ of path-connected spaces is a relative Sullivan algebra
 \[
    \iota\colon (A,\dd'')\lra ((\Lambda V)\otimes A,\dd')
 \]
with $(\Lambda V,\dd)$ a Sullivan model of $F$ such that the inclusion $\iota$ is a rational model of $E\to X$. Relative Sullivan models are cofibrant in the category of commutative differential graded algebras over a fixed commutative differential graded algebra.

\subsection{Rational models of homogeneous spaces}\label{ss:homogeneous} 

\cite[pp.\, 215--218]{FHT} and \cite[Theorem~2.71]{FOT} Let $G$ be a connected Lie group and $H\subset K$ a closed subgroup. Let $\Lambda V$ and $\Lambda W$ (with trivial differential) be minimal Sullivan models of $BH$ and $BG$, and let $\varphi\colon \Lambda W\to \Lambda V$ be the rational model of the induced map $BH\to BG$. Then the Sullivan algebra $\Lambda V\otimes \Lambda s^{-1}W$ with differential $\dd(s^{-1} w) = \varphi(w)$ if $w\in W$ and $\dd(v) = 0$ if $v\in V$ is a Sullivan model of $H\backslash G$. 
\begin{lemma}
    The relative Sullivan algebra $\iota\colon \Lambda W\to\Lambda V\otimes \Lambda s^{-1}W \otimes \Lambda W$ with differential $\dd(s^{-1}w) = \varphi(w)-w$ if $w\in W$ and $\dd(v) = \dd(w) = 0$ if $v\in V$ or $w\in W$ is a model of the fibration $H\backslash G\times_G EG\to BG$.    
\end{lemma}
\begin{proof}
    The fibration $H\backslash G\times_G EG\to BG$ is homotopy equivalent to the map $BH\to BG$, thus it is enough to build a quasi-isomorphism $\eta$ of Sullivan algebras making the following diagram commute:
    \[\xymatrix{
    \Lambda V &\ar@{-->}[l]_-{\eta}\Lambda V\otimes \Lambda s^{-1}W \otimes \Lambda W \\
    \Lambda W\ar[u]^{\varphi}\ar[ur]_{\iota}&
    }\]
    The quasi-isomorphism sends an $v\otimes 1\otimes 1$ to $v$, sends $1\otimes s^{-1}w\otimes 1$ to $0$ and $1\otimes 1\otimes w$ to $\varphi(w)$.
\end{proof}
In particular, if $H$ is the trivial subgroup, we obtain the relative Sullivan model
\[
    \Lambda s^{-1} W\otimes \Lambda W\longleftarrow \Lambda W
\]
with differential $\dd(s^{-1}w) = -w$ of the universal right $G$-fibration
$
   EG\lra BG.
$

	\subsection{Models of non-connected spaces} Following \cite{BuijsMurillo:basic}, let $X$ be a rational space with nilpotent connected components of finite type. Let $(A,d)$ be a commutative differential graded algebra concentrated in non-negative degrees. We will say that $(A,d)$ is a rational model of $X$ if for each choice of a basepoint $*\to X$, there is an augmentation $\epsilon\colon A\to \bQ$ such that the quotient $(A\oplus \bQ)/a\sim \epsilon(a)$ is a model of the connected component of $X$ that contains the basepoint $*$.

	\subsection{Models of section spaces}
	Let $\pi\colon E\to X$ be a fibration with finite dimensional connected base $X$ and nilpotent finite type fiber $F$. 	The space of sections of this bundle has nilpotent components \cite{Moller:nilpotent}, and therefore is suitable to be studied through the rational homotopy models of Quillen and Sullivan. Suppose that the fibration has a preferred section $f_0\colon X\to E$, and let $f$ be another section. Let $(B,d)$ be a strong rational model of $X$ (not necessarily minimal nor free) and let $A = (\Lambda V,d)$ be a Sullivan model of $F$ (not necessarily minimal). A \emph{relative model} of the fibration $\pi$ is a model of $E$ of the form $(\Lambda V\otimes B,d)$ such that
 \begin{itemize}
     \item [(i)] the map $B\to \Lambda V\otimes B$ given by $b\mapsto 1\otimes b$ is a model of the projection $\pi$,
     \item [(ii)] the projection $\Lambda V\otimes B\to \Lambda V$ that is given on generators by sending $v\otimes b$ to zero unless $b=1$ in which case it is sent to $v$, is a model of the inclusion of the fibre.
 \end{itemize}

 Let $(B_*,\ddual)$ be the dual coalgebra of $B$, with basis $\bar{B}_*$. There are two possible conventions for the differential in this dual coalgebra: either it is the dual $\ddual$ of the differential or $f\mapsto (-1)^{|f|+1}\ddual f$ (as in \cite{Lawson:signs}). Following the literature on rational homotopy models of mapping spaces, we have chosen the first convention. Observe also that after dualization, the graded coalgebra $B_*$ sits in non-positive cohomological degrees.
 
 Consider first the chain complex $A\otimes B\otimes B_*$, of which we take the free graded algebra $\Lambda(A\otimes B\otimes B_*)_{\geq 0}$. On this cdga, consider the following ideals:
	\begin{align*}
		J &= \langle b\otimes \beta - (-1)^{\alpha(|b|)}\beta(b)\mid b\in B,\beta\in B_*\rangle
		\\
		I &= \langle a_1a_2\otimes \beta - \sum_j (-1)^{|a_2||\beta_j'|}(a_1\otimes \beta_j') (a_2\otimes \beta_j'')\mid a_1,a_2\in A, \beta\in B_*\rangle
	\end{align*}
	where $\Delta(\beta) = \sum_j \beta'_j\otimes \beta''_j$ is the comultiplication in $B_*$, $\alpha(r) = \lfloor\frac{r+1}{2}\rfloor$ and $\beta(b)$ is the duality pairing $B_*\otimes B\to \bQ$.

	\begin{proposition}\label{prop:BS}\cite[Theorems~4.4\& 6.1]{BS:models} (cf. \cite[Prop.~4.4]{BuijsMurillo:basic})
		The inclusion of graded commutative algebras
  $$
    \Lambda(V\otimes B_*)\lra\Lambda(A\otimes B\otimes B_*)
$$
induces on the quotient
  $$
    \Lambda(V\otimes B_*)\lra\Lambda(A\otimes B\otimes B_*)/(J+I)
$$
an isomorphism of graded commutative algebras. This isomorphism endows $\Lambda(V\otimes B_*)$ with a differential and there is a map
\[
    \langle\Lambda(V\otimes B_*)\rangle\lra \Gamma(E_\bQ\to X_\bQ).
\]
The preferred section $f_0\colon X\to E$ gives a $0$-simplex on the right, which corresponds to an augmentation $\epsilon\colon \Lambda(V\otimes B_*)\to \bQ$. Let $K$ be the ideal generated by all the elements in negative degrees, their images under the differential, and the elements of the form $\epsilon(u)-u$. The quotient
$$
    \Lambda(V\otimes B_*)/K
$$
is a Sullivan model of the connected component of the space of sections of the fibration $E\to X$ that contains the section $f_0$.
	\end{proposition}

 \begin{remark}
    The sign $(-1)^{\alpha(|b|)}$ that appears in the ideal $J$ is due to the convention we have chosen for the differential in the dual coalgebra. Had we dualized as in \cite{Lawson:signs}, this sign would have disappeared. We have not tried that convention to stay compatible with all the previous literature on rational models of spaces of sections.
\end{remark}

\subsection{The rational model of the action by conjugation on the space of sections}
The most convenient way to encode an action up to homotopy of a path-connected Lie group $G$ on a path connected nilpotent space $X$ in rational homotopy theory is to give a relative rational model of the Borel fibration
\[
	X\overset{i}{\lra} X\times_{G} EG\overset{p}{\lra} BG,
\]
i.e., a model of $X\times_G EG$ of the form $(\Lambda V\otimes \Lambda W,d)$, where $\Lambda V$ is a model of the fibre $X$ and $\Lambda W$ is a model of the base $BG$, together with maps
\[
    \Lambda W\overset{p}{\lra} (\Lambda V)\otimes (\Lambda W)\overset{i}{\lra} \Lambda V
\]
with $p(a) = a\otimes 1$ and $i$ is the quotient map $(\Lambda V)\otimes (\Lambda W) \to (\Lambda V)\otimes (\Lambda W) / (\Lambda W)^{>0}\cong \Lambda V$. In particular, if $a\in \Lambda W$ then $\dd(1\otimes a) = 1\otimes da$, while if $v\in V$, then $\dd(v\otimes 1) = \dd(v)\otimes 1 + \ldots$ with the remaining terms of the form $v'\otimes a$ with $a\neq 1$.

Suppose that the fibration $F\to E\to X$ in the previous section is endowed with an action of a compact connected topological group $G$, in the sense that both the total space and the base are acted on by $G$ and the projection $E\to X$ is equivariant. Suppose we are given a relative model $B\to (\Lambda V)\otimes B$ with $B = \Lambda U$ of the fibration as in the previous section, and, additionally, a twice relative model $(\Lambda V\otimes B)\otimes  \Lambda W$ of the fibration $E\times_GEG\to BG$, i.e., a commutative triangle of relative Sullivan algebras
\begin{equation}\label{eq:010}
    \begin{gathered}
    \xymatrix{
        B\otimes \Lambda W \ar[r] & (\Lambda V) \otimes B \otimes \Lambda W \\
        \Lambda W\ar[u]\ar[ur]
    }
    \end{gathered}
\end{equation}
that models the triangle of fibrations
\begin{equation}\label{eq:011}
\begin{gathered}
    \xymatrix{
        X\times_{G} EG \ar[d] & E\times_{G} EG \ar[dl]\ar[l]\\
        BG
    }
\end{gathered}
\end{equation}
Recall that, under these conditions, $G$ acts on the space of sections by conjugation: if $g\in G$ and $f$ is a section, then $(g(f))(x) = g(f(g^{-1}(x)))$. Let $\bar{W}$ be an additive basis of $\Lambda W$ and let $\omega$ denote the dual of $w\in \bar{W}$. 

\begin{remark}
    Observe that the cofibre of the vertical maps in \eqref{eq:010} is $B\to (\Lambda V)\otimes B$, which is a model of the fibre $E\to X$ of the vertical maps in \eqref{eq:011}. Similarly, the cofibre of the map $B\to (\Lambda V)\otimes B$ is a model $\Lambda V$ of the fibre $F$ of the map $E\to X$. Finally, observe that, since $\Lambda W\to B\otimes \Lambda W$ is a relative Sullivan algebra, we are implicitely stating that $B$ is also a Sullivan algebra. 

    In each of these Sullivan algebras, let us use the following notation for the differential:
    \begin{align*}
        ((\Lambda V)\otimes B\otimes \Lambda W,\partial)
        &&
        (\Lambda V\otimes B,\partial_E)
        &&
        (\Lambda V,\partial_F)
        &&
        (B,\partial_X)
        &&
        (\Lambda W,\partial_{BG}).
    \end{align*}
Then,
    \[
        \partial(v\otimes 1\otimes 1) = \partial_E(v)\otimes 1 + \ldots
    \] 
    where the remaining terms are in $(\Lambda V)\otimes B\otimes (\Lambda W)^{>0}$, 
    \[
        \partial_E(v\otimes 1) = \partial_F(v)\otimes 1+\ldots
    \]
    where the remaining terms are in $(\Lambda V)\otimes (B)^{>0}$,
    \[
        \partial(1\otimes b\otimes 1) = 1\otimes \partial_X(b)\otimes 1 + \ldots
    \]
    where the remaining terms are in $1\otimes B\otimes (\Lambda W)^{>0}$, and
    \[
        \partial(1\otimes 1\otimes w) = 1 \otimes 1\otimes \partial_{BG}(w).
    \]
\end{remark}
\begin{theorem} \label{thm:newBS}
    Suppose a commutative triangle of relative Sullivan algebras \eqref{eq:010} modelling the triangle of fibrations \eqref{eq:011} is given. The Borel construction of the action of $G$ by conjugation on the space of sections of the fibration $E\to G$ has the following relative model:
\[
    \Lambda W\lra \Lambda (V\otimes B_*)/K \otimes \Lambda W
\]
with differential
\[
    \dd(v\otimes \beta) = \dd v\otimes \beta + (-1)^{|v|}\sum_{w\in \bar{W}} (-1)^{\alpha(|b|)+\alpha(|w|+|b|)} v\otimes \ddual(\beta\otimes \omega)\otimes w
\]
where $(\dd v)\otimes \beta\in \Lambda (\Lambda (V)\otimes B\otimes B_*)/ (I,J) \otimes \Lambda W \cong \Lambda(V\otimes B_*)\otimes \Lambda(W)$, and we are abusing notation by writing $$\ddual\colon B_*\otimes (\Lambda W)^{\vee} \to B_*$$ for the dual of the composition $B\otimes 1 \to B\otimes \Lambda W\overset{\partial}{\to} B\otimes \Lambda W$.
\end{theorem}
\begin{proof}
After checking that $\Lambda (V\otimes B_*)\otimes \Lambda W$ is a well-defined Sullivan algebra, we have that the relative model
    \begin{equation}\label{eq:10}
        \Lambda W\lra (\Lambda(V\otimes B_*)/K)\otimes \Lambda W
    \end{equation}
    is a model for the fibration associated to some action of $G$ on the space of sections $\Gamma(E_\bQ\to X_{\bQ})$. 

    Now, using the model $(\Lambda W)\otimes B\otimes \Lambda W$ of $E\times_{G} EG$ from \eqref{eq:010}, consider the morphism of cdgas
    \[
        h\colon (\Lambda V)\otimes B\otimes \Lambda W\lra ((\Lambda(V\otimes B_*)/K)\otimes \Lambda W)\otimes_{\Lambda W} (B\otimes \Lambda W) \cong (\Lambda(V\otimes B_*)/K)\otimes B\otimes \Lambda W
    \]
    that restricts to the identity on $1\otimes B\otimes  \Lambda W$ and sends a generator $v\otimes 1\otimes 1$ to $\sum_{b\in \bar{B}} (-1)^{\alpha(|b|)}(v\otimes \beta)\otimes b\otimes 1$, where $\bar{B}$ is an additive basis of $B$ and $\beta$ denotes the dual of $b\in \bar{B}$. In order to check that it commutes with the differential, observe that if $v\otimes 1\otimes 1\in \Lambda V\otimes \Lambda W$, then the differential of its image is
    \begin{align*}
        &\sum_b (-1)^{\alpha(|b|)} \partial v\otimes \beta\otimes b\otimes 1 \\ 
        +&\sum_{b,w} (-1)^{\alpha(|b|+|w|)}(-1)^{|v|} v\otimes \ddual (\beta\otimes \omega)\otimes b\otimes w
    \\ +&\sum_b (-1)^{\alpha(b)}(-1)^{|v|+|\beta|} v\otimes \beta\otimes\dd(b).
    \end{align*}
    Of these three terms, the first is precisely the image of the differential of $v\otimes 1\otimes 1$, and the other two cancel as follows: if the generator $b'\otimes w'$ appears in $\partial(b)$ with coefficient $c_{b',w'}\in \bQ$, then $\beta$ will appear with the same coefficient in $\ddual(\beta'\otimes \omega')$. Thus the sign difference between their appeareances in the last two terms is $\alpha(|b'|+|w'|) + \alpha(|b|) + |\beta|$. If $|\beta| = r$, then this sum is $\lfloor \frac{r+2}{2}\rfloor + \lfloor\frac{r+1}{2}\rfloor+ r $ which is always odd, thus these terms do cancel. 
    
    This gives a diagram of relative Sullivan algebras
    \[
        \xymatrix@C=10pt{
            \Lambda(V\otimes B_*)/K\otimes B && \Lambda(V\otimes B_*)/K\otimes B\otimes \Lambda W\ar[ll] && \\
            &B\ar[ul]\ar[dl] && B\otimes \Lambda W \ar[ul]\ar[dl]\ar[ll]& \Lambda W\ar[l] \ar[ull]\ar[dll]\\
            \Lambda V\otimes B \ar[uu] && \Lambda V\otimes B \otimes \Lambda W\ar[uu]^>>>>>{h}\ar[ll] &&
        }
    \]
    and the left vertical map is the model of the evaluation map given in \cite[Theorem~1.1]{BuijsMurillo:basic}. After realization, we obtain the following diagram:
    \[
        \xymatrix{
            \Gamma(E)\times X \ar[rr]\ar[dd]_{\mathrm{ev}} \ar[dr] && (\Gamma(E)\times X) \times_G EG \ar[dr]\ar[dd] \ar[drr] && \\
            & X \ar[rr] && X\times_G EG \ar[r] & BG \\
            E\ar[ur]\ar[rr] && E\times_G EG \ar[ur]\ar[urr]
        }
    \]
    where the upper sequence is the Borel construction for some action of the group $G$ on the space of sections. Taking the adjoints of the vertical maps we obtain
        \[
        \xymatrix{
            \Gamma(E) \ar[rr]\ar[dd]_{\Id}  && \Gamma(E) \times_G EG \ar[dd] \ar[drr] && \\
            & &&& BG \\
            \Gamma(E)\ar[rr] && \Gamma(E)\times_G EG \ar[urr]
        }
    \]
    where the action of the group $G$ in the bottom row is again by conjugation. As the evaluation map is adjoint to the identity map, the left vertical arrow is a homotopy equivalence, and, therefore, so is the middle vertical map. As a consequence, the action of $G$ on the top row is also the action by conjugation.
\end{proof}
While this article was being refereed we learned from Nils Prigge that this result can be found in Section 5 of \cite{Haefliger-GF}. We find our presentation of the space of sections more explicit than the one given there, so we have decided to keep it here.


\section{The action of $\PU(n+1)$ on the space of sections}
\def\bP{\mathbb{P}}
\def\cO{\mathcal O}
We focus now on the action of $PU(n+1)$ on the bundle $\PCbun\to \CP^n$. Since $T\CP^n\oplus \epsilon = \cO(1)^{\oplus n+1}$ and $\bP(E\otimes L)\cong \bP(E)$ when $E$ is a vector bundle and $L$ is a line bundle, we have that
\[
    \bP(T\CP^n\oplus \epsilon)\cong \bP(\cO(1)^{\oplus n+1}) \cong \bP((\cO(1)^{\oplus n+1})\otimes \cO(-1)) \cong \bP(\epsilon^{n+1}).
\]
Thus the bundle $\PCbun\to \CP^n$ is isomorphic to the trivial bundle and therefore there is a homeomorphism
\begin{equation}\label{eq:2}
    \Gamma(\PCbun)\lra \map(\CP^n,\CP^n)
\end{equation}
that takes sections of degree $\degree$ to maps of degree $\degree - 1$. On the other hand, we are interested in the fibration\[
    \PCbun\times_{U(n+1)} EU(n+1)\lra \CP^n\times_{U(n+1)} EU(n+1). 
\]
which encodes the action of $U(n+1)$. The second space is homotopy equivalent to $BU(1)\times BU(n)$ and, writing $\gamma_1$ and $\gamma_n$ for the tautological bundles, the fibration becomes $\P(\gamma_1^\vee\otimes \gamma_n\oplus \epsilon)$ which is not trivial. As a consequence, the action of $U(n+1)$ on $\Gamma(\PCbun)_\degree$ is not comparable a priori with the action of $U(n+1)$ on $\map(\CP^n,\CP^n)_{d-1}$.

The rational homotopy type of spaces of maps between projective spaces was determined in \cite{Moller-Raussen}, and from the homeomorphism \eqref{eq:2} we obtain that the rational cohomology of $\Gamma(\PCbun)_\degree$ is isomorphic to the rational cohomology of $PU(n+1)$ if $\degree \neq 1$. In this section we first describe the fibrewise complex Grassmannian $\Gr_1^{\bC}(T\CP^n)$ as a homogeneous space and find a Sullivan model of the Borel construction of the action of $\U(n+1)$ on it. Second, we find a Sullivan model of the action of $\U(n+1)$ on the fibration $\Gr_1^{\bC}(T\CP^n)\to \CP^n$. Then, we find a Sullivan model of the action of $\U(n+1)$ on the fibrewise Thom space of the fibrewise tautological bundle over $\Gr_1^{\bC}(T\CP^n)\to \CP^n$, that we denote $\Cbun\to \CP^n$. This fibre bundle is homeomorphic to $\PCbun$. Finally, we obtain a Sullivan model of the action of $\U(n+1)$ on the space of sections and deduce Theorem \ref{thm:main}.

\begin{remark}
    We will work with the bundle $\Cbun$ instead of the homeomorphic bundle $\PCbun$ so that our constructions are more closely related to the scanning map. We expect this may be helpful in order to compare the space of complex hypersurfaces to other spaces of submanifolds that admit scanning maps ---for example, codimension $2$ oriented smooth submanifolds.
\end{remark}

\subsection{A homogeneous construction of \texorpdfstring{$\Cbun$}{P1(TCPn)}}

The projective space $\CP^n$ is obtained as the homotopy fibre of the map $\BU(1)\times \BU(n)\to \BU(n+1)$ given by Whitney sum, in particular $\BU(1)\times \BU(n)$ is a model of the Borel construction $\CP^n\times_{\U(n+1)} \EU(n+1)$. Additionally, the pullback along $\CP^n\to \BU(1)\times \BU(n)$ of the two tautological bundles yield the tautological bundle of $\CP^n$ and its complement. We will need a description of $\CP^n$ in which the pullback of the map to $\BU(n)$ yields the tangent bundle instead of the complement of the tautological bundle. 

Recall that the tangent space of $\CP^n$ is isomorphic to the bundle $\hom(\gamma_1,\gamma_1^\perp)$ and consider the composition
\[
	\BU(1)\times \BU(n)\overset{\Delta\times \Id}{\lra} \BU(1)\times \BU(1)\times \BU(n) \overset{\Id\times \otimes}{\lra} \BU(1)\times \BU(n)\overset{\oplus}{\lra} \BU(n+1)
\]
where $\Delta$ is the diagonal and $\otimes$ classifies the tensor product of complex vector bundles. Let $q$ be the composition of all but the last map and let $f$ be the whole composition. Then $q$ is invertible, its inverse being the composition
\[\xymatrix{
	\BU(1)\times \BU(n)\ar[r]^-{\Delta\times \Id}& 
	\BU(1)\times \BU(1)\times \BU(n)	\ar[d]^{\Id\times D \times \Id}&
	\\
	&
	\BU(1)\times \BU(1)\times \BU(n)\ar[r]^-{\Id\times \otimes} &
	\BU(1)\times \BU(n),
}
\]
where $D$ classifies duality of complex vector bundles. Then, we have the following homotopy pullback squares:
\[
    \xymatrix{
        \Gr_1^{\bC}(T\CP^n)\cong \frac{\U(n+1)}{\U(1)\times \U(1)\times \U(n-1)} \ar[r]
 \ar[d] & BU(1)\times BU(1) \times \BU(n-1) \ar[d]^{\Id\times \oplus} \\
    \CP^n\ar[r]\ar[d]& BU(1)\times BU(n)\ar[d]^f  \\
    \ast\ar[r] & BU(n+1)
 }
\]
where $\CP^n$ is obtained as the homotopy fibre of $f$, the fibre of the right vertical composition is $\Gr_1^{\bC}(T\CP^n)$ and the fibre of the classifying map $\oplus$ is the projective bundle of the tautological bundle $\gamma_n$ over $\BU(n)$. Thus, the pullback of the tautological bundle over $\BU(n)$ along the inclusion of the homotopy fibre $\CP^n\to \BU(1)\times \BU(n)$ is now the tangent bundle of $\CP^n$.

\subsection{A Sullivan model of \texorpdfstring{$\Gr_1^{\bC}(T\CP^n)$}{Gr1(TCPn)}}\label{s:32} By the previous discussion, we have the following presentation of the Grassmannian $\Gr_1^\bC(T\CP^n)$ of complex $1$-planes in the tangent bundle of $\CP^n$ as a homogeneous space:
\[
	\Gr_1^\bC(T\CP^n)\cong \frac{\U(n+1)}{\U(1)\times \U(1)\times \U(n-1)}
\]
where the action corresponds to the inclusion $\U(1)\times \U(1)\times \U(n-1)\hookrightarrow \U(n+1)$ classified by the map
\[
	h\colon \BU(1)\times \BU(1)\times \BU(n-1)\overset{\Id\times \oplus}{\lra} \BU(1)\times \BU(n) \overset{f}{\lra} \BU(n+1).
\]
Recall now that a minimal Sullivan model of $\U(n)$ is given by $\Lambda(s^{-1}c_1,\ldots,s^{-1}c_{n})$, where $c_{i}\in H^{2i}(\BU(n))$ is the $i$th
Chern class and $s^{-1}c_{i}\in H^{2i-1}(\U(n))$ its looping. We will denote:
\begin{itemize}
	\item $c_{i}$ the $i$th Chern class of $\BU(n+1)$.
	\item $\hat{c}_{1}$ the first Chern class of the first copy of $\BU(1)$.
	\item $\check{c}_{1}$ the first Chern class of the second copy of $\BU(1)$.
	\item $\bar{c}_{i}$ the $i$th Chern class of $\BU(n-1)$.
\end{itemize}
The classes $\check{c}_1$ and $\hat{c}_1$ will be also denoted $\check{c}$ and
$\hat{c}$. We write $\Lambda W = \Lambda (c_1,\ldots,c_{n+1})$ for the Sullivan model of $\BU(n+1)$. Using Section \ref{ss:homogeneous} and the
formula in~\cite[Example 3.2.2]{Fulton} for the Chern classes of a tensor product with a line bundle, we obtain the map induced by $h$ on the cohomology of the classifying spaces (which we also denote by $h$):
\begin{lemma}
	The map induced by $h$ is given by
	\[
		h(c_{p}) = \sum_{i=0}^p \binom{n+1-i}{p-i} \left(\hat{c}^{p-i}\otimes 1\otimes \bar{c}_{i} + \hat{c}^{p-i}\otimes \check{c}\otimes \bar{c}_{i-1}\right).
	\]
\end{lemma}
Therefore, a rational model of the homogeneous space is
\[
		\Lambda(s^{-1}c_1,\ldots,s^{-1}c_{n+1})\otimes \Lambda(\hat{c},\check{c},\bar{c}_1,\ldots,\bar{c}_{n-1})
\]
with non-trivial differential given by $\dd(s^{-1}c_{p}) = h(c_{p})$. A relative rational model of the Borel construction of the action of $\U(n+1)$ on this homogeneous space is given by
\begin{equation}\label{eq:083}
		\Lambda(s^{-1}c,\ldots,s^{-1}c_{n+1})\otimes \Lambda(\hat{c},\check{c},\bar{c}_1,\ldots,\bar{c}_{n-1})\otimes \Lambda W
\end{equation}
with non-trivial differential $\dd(s^{-1}c_{p}) = h(c_{p}) - c_{p}$. These models are not minimal, since the differential of $s^{-1}c_{p}$ in the first case maps non-trivially to the space of indecomposable elements when $p = 1,\ldots,n-1$. Specifically, the coefficient of $\bar{c}_{p}$ in $\dd(s^{-1}c_{p})$ in these cases is $1$. Therefore, we may remove the generators $s^{-1}c_1,\ldots,s^{-1}c_{n-1}$ and quotient by the relations
\begin{align*}
	\bar{c}_{p} &= -1\otimes \check{c}\otimes \bar{c}_{p-1}-\sum_{i=0}^{p-1} \binom{n+1-i}{p-i} \left(\hat{c}^{p-i}\otimes 1\otimes \bar{c}_{i} + \hat{c}^{p-i}\otimes \check{c}\otimes \bar{c}_{i-1}\right)
\\
	\bar{c}_{p} &= c_{p}-1\otimes \check{c}\otimes \bar{c}_{p-1}-\sum_{i=0}^{p-1} \binom{n+1-i}{p-i} \left(\hat{c}^{p-i}\otimes 1\otimes \bar{c}_{i} + \hat{c}^{p-i}\otimes \check{c}\otimes \bar{c}_{i-1}\right)
\end{align*}
in each case. One can prove by induction, using the combinatorial identities
\begin{align*}
	\binom{n}{k}\binom{n-k}{j} &= \binom{n}{k+j}\binom{k+j}{j} & \sum_{i=0}^p (-1)^i\binom{p}{i} &= 0
\end{align*}
that, when $p<n$, in the non-relative case we have
\[
	\bar{c}_{p} = (-1)^p\sum_{j=0}^p \binom{n+1}{j}\hat{c}^j\check{c}^{p-j},
\]
while in the relative case,
\[
	\bar{c}_{p} = (-1)^p\sum_{q=0}^p(-1)^qc_{q}\sum_{i=0}^{p-q}\binom{n-q+1}{i}\hat{c}^i\check{c}^{p-q-i}.
	\]
With these identifications, we have in the non-relative case
\begin{align}\label{eq:diff-1}
	\dd(s^{-1}c_{n}) &= (-1)^n\sum_{j=0}^n \binom{n+1}{j}\hat{c}^j\check{c}^{n-j} \\ \label{eq:diff-2}
	\dd(s^{-1}c_{n+1}) &= (-1)^{n+1}\left(n\hat{c}^{n+1}+\sum_{j=0}^{n-1} \binom{n+1}{j}\hat{c}^{j+1}\check{c}^{n-j}\right)
\end{align}
and in the relative case
\begin{align}\label{eq:diff-3}
	\dd(s^{-1}c_{n}) &= \sum_{q=0}^n(-1)^{n+q-1}c_{q}\sum_{j=0}^{n-q} \binom{n-q+1}{j}\hat{c}^j\check{c}^{n-q-j}
        \\ \label{eq:diff-4}
        \dd(s^{-1}c_{n+1}) &= \sum_{q=0}^{n+1} (-1)^{n+q+1}c_{q}\left((n-q)\hat{c}^{n-q+1} + \sum_{j=0}^{n-q-1} \binom{n-q+1}{j}\hat{c}^{j+1}\check{c}^{n-q-j}\right)
        \\ \notag
        &= \hat{c}\cdot \dd(s^{-1}c_{n}) - \sum_{q=0}^{n+1} (-1)^{n+q+1} c_{q}\hat{c}^{n-q+1}.
\end{align}
We summarize:
\begin{proposition}
	The minimal Sullivan model of $\Gr_1^{\bC}(T\CP^n)$ is 
	\[
		\Lambda(s^{-1}c_{n},s^{-1}c_{n+1},\check{c},\hat{c})
	\]
with the differential in \eqref{eq:diff-1} and \eqref{eq:diff-2}. The relative
minimal model of the fibration $\Gr_1^{\bC}(T\CP^n)\times_{\U(n+1)} \EU(n+1)\to \BU(n+1)$ is
	\[
		\Lambda(s^{-1}c_{n},s^{-1}c_{n+1},\check{c},\hat{c})\otimes \Lambda W.
	\]
with the differential in \eqref{eq:diff-3} and \eqref{eq:diff-4}.
\end{proposition}

\subsection{A relative model of
  \texorpdfstring{$\Gr_1^{\bC}(T\CP^n)\to \CP^n$}{Gr1(TCPn) to CPn}}
We now obtain a model of the fibration 
	\[
		\Gr_1^{\bC }(T\CP^n)\lra \CP^n
	\]
	relative to the action of $\U(n+1)$ from the model of the previous section, through the change of variables (their degrees are reminded in the subscripts):
	\begin{align*}
		  b_2 &= \hat{c}
		&
		a_2 &= \check{c}
		&
		x_{2n-1} &= (-1)^n s^{-1}c_{n}
		&
		y_{2n+1} &= (-1)^n(s^{-1}c_{n+1} - \hat{c}\cdot s^{-1}c_{n}).
	\end{align*}
\begin{proposition} A relative minimal model of the Borel construction of the action of $\U(n+1)$ on the fibration $\Gr_1^{\bC}(T\CP^n)\to \CP^n$ is:
\[
	\Lambda(b,y)\otimes \Lambda W\lra \Lambda(a,x,b,y)\otimes \Lambda W
\]
with differential
\begin{align*}
	\dd(x) &= \sum_{q=0}^n(-1)^{q-1}c_{q}\sum_{j=0}^{n-q} \binom{n-q+1}{j}b^ja^{n-q-j}
	\\
	\dd(y) &= \sum_{q=0}^{n+1}(-1)^q c_{q}b^{n+1-q}
\end{align*}
\end{proposition}
\begin{proof}
    First, observe that the change of variables yields the model $\Lambda(a,x,b,y)\otimes \Lambda W$ of $\Gr_1^{\bC}(T\CP^n)\times_{U(n+1)} EU(n+1)$ given in the statement.

    Second, let us denote by 
    \[
        V = \bQ\langle s^{-1}c_1,\ldots,s^{-1}c_{n+1},\hat{c},\check{c},\bar{c}_1,\ldots,\bar{c}_{n+1}\rangle.
    \]
    the vector space of generators of the algebra \eqref{eq:083}. Thinking the fibration $\Gr_1^{\bC}(T\CP^n)\to \CP^n$ as the map of homogeneous spaces
    \[
        \frac{U(n+1)}{U(1)\times U(1)\times U(n-1)}\lra \frac{U(n+1)}{U(1)\times U(n)} 
    \]
    induced by the map $\Id\times \oplus\colon U(1)\times U(1)\times U(n-1)\to U(1)\times U(n)$, we deduce that a Sullivan model of the Borel construction $\CP^n\times_{U(n+1)}EU(n+1)\to BU(n+1)$ is the subalgebra of \eqref{eq:083} generated by the vector space $V'\oplus W$ where 
    \begin{align*}
    V' &= \bQ\langle s^{-1}c_1,\ldots,s^{-1}c_{n+1},\hat{c},\tilde{c}_1,\ldots,\tilde{c}_{n}\rangle
    \end{align*}
    with 
    \begin{align*}
        \tilde{c}_1 &= \bar{c}_1 +\check{c} \\
        \tilde{c}_2&=\bar{c}_2 + \bar{c}_1\check{c}\\
        &\ldots \\
        \tilde{c}_{n-1} &= \bar{c}_{n-1} + \bar{c}_{n-1}\check{c} \\ \tilde{c}_n &= \bar{c}_{n-1}\check{c}
    \end{align*} 
    Working as in Section \eqref{s:32} (with simpler calculations), we obtain that the relative minimal model associated to that subalgebra is isomorphic to $\Lambda(b,y)\otimes \Lambda W$ with the differential in the statement.

    Finally, a relative Sullivan model of the projection 
    \[
    \Cbun\times_{U(n+1)} EU(n+1)\to \CP^n\times_{U(n+1)}EU(n+1)
    \]
    is a homotopy lift of the composition
    \[\xymatrix{
    && \Lambda(a,x,b,y)\otimes \Lambda W\ar[d]^{\simeq} \\
        \Lambda(b,y)\otimes \Lambda W\ar[r]^{\simeq}\ar@{-->}[urr] & \Lambda V'\otimes \Lambda W \ar[r] & \Lambda V\otimes \Lambda W.
    }\]
    Since the horizontal composition maps $b$ to $\hat{c}$, whose only lift is $b$, we deduce that the diagonal map sends $b$ to $b$. We then conclude noting that only choice that makes this a map of differential graded algebras is mapping $y$ to $y$.
\end{proof}

\subsection{A relative model of \texorpdfstring{$\Cbun\to \CP^n$}{P1(TCPn) to CPn}} To obtain a relative model of the Thom space, it is enough to restrict the previous model to the ideal generated by the Euler class of the tautological bundle, which in this case is $a$ (see \cite[Thm.\, 2.4]{FOT:rhtThom} or \cite[Thm.\, 2.9]{BCC:rhtThom}). This ideal is generated as a subalgebra by $u=a$ and $t = ax$, and we obtain the following model:
\begin{proposition}\label{prop:35}
    A relative Sullivan model of the projection $\Cbun\to \CP^n$ is
    \[
	\Lambda(b,y)\lra \Lambda(u,t,b,y)
\] 
with differential
\begin{align*}
	\dd(t) &= (-1)^n\sum_{i=0}^n \binom{n+1}{i} b^{i}u^{n-i+1} \\
	\dd(y) &= b^{n+1}.
\end{align*}
    A relative Sullivan model of the Borel construction of the action of $\U(n+1)$ on the projection $\Cbun\to \CP^n$ is
    \[
	\Lambda(b,y)\otimes \Lambda W\lra \Lambda(u,t,b,y)\otimes \Lambda W
\] 
with differential
\begin{align*}
	\dd(t) &= \sum_{q=0}^n(-1)^{q-1}c_{q}\sum_{j=0}^{n-q} \binom{n-q+1}{j}b^ju^{n-q-j+1}
        \\
        \dd(y) &= \sum_{q=0}^{n+1}(-1)^q  c_{q}b^{n+1-q}.
\end{align*}
\end{proposition}
\begin{proof}
     For the relative model of the Borel construction, we redo the argument in \cite{FOT:rhtThom} and \cite{BCC:rhtThom} fibrewise: If $F\to E\overset{p}{\to} B$ is a fibre bundle and $F'\to E'\overset{p'}{\to} B$ is a subbundle, consider its fibrewise cofiber. 
\[
	\xymatrix{
		F' \ar[r]^{\subset}\ar[d] & F\ar[d] \ar[r] & F/F' \ar[d] \\
		E' \ar[r]^{\subset} \ar[dr] & E\ar[d]\ar[r] & E'' \ar[dl] \\
		&B&
	}
\]
If we are given relative Sullivan models of the fibre bundles $p,p'$ and the inclusion:
\[
	\xymatrix{
		\Lambda V' & \Lambda V \ar[l]^{\iota} \\
		\Lambda W \otimes \Lambda V'\ar[u] & \Lambda W\otimes \Lambda V\ar[l]\ar[u] \\
		&\Lambda W\ar[ul]\ar[u]
	}
\]
such that $\iota$ is surjective, then the relative Sullivan algebra
\[
	\operatorname{Ker} \iota \longleftarrow \Lambda W\otimes \operatorname{Ker} \iota \longleftarrow \Lambda W
\]
is a relative model (in general not a Sullivan model) of the fibrewise cofiber $F/F'\to E''\to B$, i.e., there is a commutative diagram
\begin{equation}\label{eq:1111}
    \begin{aligned}
	\xymatrix{
		\operatorname{Ker} \iota & \ar[l]\Lambda W\otimes \operatorname{Ker} \iota &\ar[l]  \Lambda W \\
		\Lambda V'' \ar[u]^{\simeq}& \ar[l]\Lambda W\otimes \Lambda V''\ar[u]^{\simeq} &\ar[l]  \Lambda W\ar[u]^{\mathrm{id}}
	}
     \end{aligned}
\end{equation}
where $\Lambda W\to \Lambda V''\otimes \Lambda W$ is a relative Sullivan model of $F/F'\to E''\to B$. To produce the middle vertical map one notices that the composition
\[
    \Lambda W\otimes \Lambda V'\overset{\iota}{\longleftarrow} \Lambda W\otimes \Lambda V \longleftarrow \Lambda W\otimes \Lambda V''
\]
is nullhomotopic, and since the $\iota$ is surjective, we may find a map homotopic to the second map such that the composition with $\iota$ is zero, which therefore factors through $\Lambda W\otimes \ker \iota$. The map induced on fibres $\Lambda V''\to \ker \iota$ is then a quasi-isomorphism, because we already know that $\ker\iota$ is a model of the Thom space. As a consequence, the middle map in \eqref{eq:1111} is a quasi-isomorphism.

Since the Thom space of a vector bundle is the cofiber of the inclusion of the sphere bundle into the disc bundle, we can apply the construction above.

A relative Sullivan model of $S(\gamma_1)\to \Gr_1^{\bC}(\CP^n)$ is $\Lambda(a,x,b,y,w)\leftarrow \Lambda(a,x,b,y)$ with $d(w) = a$ \cite[p.\,202]{FHT}. To obtain a model of the Borel construction, we understand this sphere bundle as the map of homogeneous spaces
\[
    \frac{U(n+1)}{U(1)\times U(n-1)}\lra \frac{U(n+1)}{U(1)\times U(1)\times U(n-1)}
\]
induced by the inclusion $U(1)\times U(n-1)\to U(1)\times U(1)\times U(n-1)$ that includes the first factor into the first factor and the second factor into the third factor. From this we obtain the Sullivan model of the Borel construction on the sphere bundle:
\[
    \Lambda(b,y,a,x)/I(a)\otimes \Lambda W,
\]
where $I(a)$ is the ideal generated by $a$. Finally, the relative kernel of the surjective map
\[
    \Lambda(b,y,a,x)/I(a)\otimes \Lambda W\longleftarrow \Lambda(b,y,a,x)\otimes \Lambda W
\]
is the one proposed in the statement.
\end{proof}

\subsection{A model of the space of sections} Let $B = \bQ\langle 1,\beta_1,\beta_2,\ldots,\beta_{n}\rangle$ be the rational homology of $\CP^n$ with its coalgebra structure, with $\beta_i$ the dual of $b^i$. As $\CP^n$ is formal, this is the dual of a rational model of $\CP^n$. Following the construction of Brown and Sczcarba (Proposition~\ref{prop:BS}), we construct the algebra
\[
	\Lambda(u\otimes 1, u\otimes \beta_1,t\otimes 1,t\otimes \beta_1,\ldots,t\otimes \beta_{n})
\]
with differential
\[
	\dd(t\otimes \beta_j) = \sum_{i=0}^{n} \binom{n+1}{i} b^iu^{n-i+1}\otimes \beta_j
\]
and, after applying the relations of the ideals $I$ and $J$,
\begin{align*}
	b^{k}\otimes \beta_j
	&= \begin{cases}
		(-1)^{k+1} \beta_{j-k} & \text{if $k\geq j$} \\
		0 & \text{otherwise.}
		\end{cases}
	\\
	u^k\otimes \beta_j &= \binom{k}{j}(u\otimes \beta_1)^j(u\otimes 1)^{k-j}
\end{align*}
(the second is proven by induction again. The first uses that $\alpha(b^k) \equiv k+1\mod 2$), becomes
\[
	\dd(t\otimes \beta_j) = \binom{n+1}{j}\left(\sum_{i=0}^j (-1)^{i+1}\binom{j}{j-i}(u\otimes \beta_1)^{j-i}\right)(u\otimes 1)^{n-j+1}
\]
or,
\[
	\dd(t\otimes \beta_j) = -\binom{n+1}{j}\left(u\otimes \beta_1-1\right)^j(u\otimes 1)^{n-j+1}.
\]
If $f$ is a section, the augmentation $\epsilon_f\colon \Lambda(u\otimes 1, u\otimes \beta_1,t\otimes 1,t\otimes \beta_1,\ldots,t\otimes \beta_{n})\to \bQ$ sends $u\otimes \beta_1$ to the degree $\degree$ of $f$.
\begin{proposition}\label{prop:model-sections}
    A Sullivan model of the connected component $\Gamma(\Cbun)_{\degree}$ of sections of degree $\degree$ is 
\[
	\Lambda(u\otimes 1,t\otimes 1,t\otimes \beta_1,\ldots,t\otimes \beta_n)
\]
with differential
\[
	\dd(t\otimes \beta_j) = -\binom{n+1}{j}\left(\degree-1\right)^j(u\otimes 1)^{n-j+1}.
\]
\end{proposition}
We see that, if $\degree\neq 1$, then $\dd(t\otimes \beta_n) = -(n+1)(\degree-1)^n(u\otimes 1)$ is a nonzero multiple of $(u\otimes 1)$, and therefore each component of the space of sections has the rational homotopy type of $PU(n+1)$, with explicit cocycles in degree $2(n-j)+1$ mapping to the looping of the Chern classes:
\begin{equation}\label{eq:301}
	x_{2(n-j)+1}=(t\otimes \beta_j) - \frac{1}{n+1}\binom{n+1}{n-j}(u\otimes 1)^{n-j}(t\otimes \beta_n).
\end{equation}
If $\degree = 1$, then it has the rational homotopy type of $\U(n)\times \bC P^n$.

\subsection{A relative model of the Borel construction on the space of sections}
We now find a relative model of the Borel construction 
\[\Gamma(\Cbun)_{\degree}\times_{\U(n+1)}\EU(n+1) \to \BU(n+1).
\]
using Theorem \ref{thm:newBS}. In our case, by Proposition \ref{prop:35} we have 
\begin{align*}
    \Lambda V &= \Lambda (u,t) &
    B &= \Lambda(b,y) &
    \Lambda W &= \Lambda (c_1,\ldots,c_{n+1}).
\end{align*} 
We will denote the dual coalgebra of $\Lambda(b,y)$ as follows: The dual of $b^j$ will be denoted $\beta_{j}$ and the dual of $b^j\cdot y$ will be denoted $\gamma_{2(n+j)+1}$. We will denote the dual coalgebra of $\Lambda W$ as follows: the dual of the product $c_{i_1}^{\ell_1}\cdot\ldots \cdot c_{i_k}^{\ell_k}$ will be denoted $\theta_{i_1,\ldots,i_k}^{\ell_1,\ldots,\ell_k}$. Then the coalgebra dual to $\Lambda(b,y)$ is
\[
    B_* = \bQ\langle \{\beta_{j}\}_j, \{\gamma_{2(n+j)+1}\}_j\rangle
\]
with differential $\ddual(\beta_{n+j+1}) = \gamma_{2(n+j)+1}$. The coalgebra dual to $\Lambda(b,y)\otimes \Lambda W$ is 
\[
    B_*\otimes (\Lambda W)^{\vee} = \bQ\langle \{\beta_{j}\}_j, \{\gamma_{2(n+j)+1}\}_j\rangle\otimes \bQ\langle\{ \theta_{i_1,\ldots,i_k}^{\ell_1,\ldots,\ell_k}\}\rangle
\]
and the dual of the composition $B\to B\otimes \Lambda W\overset{\partial}{\to} B\otimes \Lambda W$ vanishes on all generators except for
\[
    \ddual(\beta_{j}\otimes \theta^1_{q}) = (-1)^q\gamma_{2(q+j)-1}.
\]
We have that
\[
    \Lambda V\otimes B_* = \Lambda(\{u\otimes \beta_{j}\}_j, \{u\otimes \gamma_{2j+1}\}_j, \{t\otimes \beta_{j}\}_j ,\{t\otimes\gamma_{2j+1}\}_j)
\]
and the non-trivial differentials in 
\[
	(\Lambda V\otimes B_*)\otimes \Lambda W
\]
in degrees $\geq -1$ are
{\footnotesize
\begin{align*}
	\dd(t\otimes \beta_{k}) &= 
	(-1)^n\sum_{q=0}^n(-1)^{q-1}c_{q}\sum_{j=0}^{n-q}(-1)^{j+1}\binom{n-q+1}{j}u^{n-q-j+1}\otimes \beta_{k-j}
	\\
        &- c_{n+1-k}\cdot t\otimes \gamma_{2n+1}
        \\
	&=(-1)^n\sum_{q=0}^n(-1)^{q-1}c_{q}\sum_{j=0}^{n-q}(-1)^{j+1}\binom{n-q+1}{j}\binom{n-q-j+1}{k-j}\degree^{k-j}(u\otimes 1)^{n-q-k+1}
 \\
        &- c_{n+1-k}\cdot t\otimes \gamma_{2n+1}
	\\
	&= (-1)^n\sum_{q=0}^n(-1)^{q-1}c_{q}\binom{n-q+1}{k}\sum_{j=0}^{n-q}(-1)^{j+1}\binom{k}{j}\degree^{k-j}(u\otimes 1)^{n-q-k+1}
 \\
        &- c_{n+1-k}\cdot t\otimes \gamma_{2n+1}.
\end{align*}
}
In the second equality we have applied the relations in the ideals $I$ and $J$ and in the third equality we have simplified the combinatorial numbers. The last term in each sum is obtained as
\[
    (-1)^{|t|}(-1)^{\alpha(|b_k|)+\alpha(|b_k|+|c_{q}|)}\partial^{\vee}t\otimes (\beta_k\otimes \theta_q^1)\otimes c_q.
\]
This summand is in degree $<-1$ unless $q = n-k+1$, in which case it has degree $0$.
Now, in dimension $0$ there are two generators: $u\otimes \beta_1$ and $t\otimes \gamma_{2n+1}$, which are related by
\[
    \dd(t\otimes \beta_{n+1}) = 1-(1-(u\otimes \beta_1))^{n+1} - t\otimes \gamma_{2n+1},
\]
therefore, if $\epsilon$ is the augmentation associated to the component that contains a section $f$, then $\epsilon(t\otimes \gamma_{2n+1}) = \epsilon(1-(1-u\otimes \beta_1)^{n+1})$. Writing $\degree = \epsilon(u\otimes \beta_1)$, we have that $\epsilon(t\otimes \gamma_{2n+1}) = 1-(1-\degree)^{n+1}$.

\begin{proposition}\label{prop:341}
    A relative model of the action of $\U(n+1)$ on the space of sections of degree $\degree$ is
    \[    
    \Lambda(u\otimes 1,t\otimes 1,t\otimes \beta_1,\ldots,t\otimes \beta_{n})\otimes \Lambda W
\]
with differential
\begin{align*}
\dd(t\otimes \beta_{k})&= (-1)^n\sum_{q=0}^n(-1)^{q-1}c_{q}\binom{n-q+1}{k}\sum_{j=0}^{n-q}(-1)^{j+1}\binom{k}{j}\degree^{k-j}(u\otimes 1)^{n-q-k+1} \\
&- (1-(1-\degree)^{n+1})c_{n-k+1}.
\end{align*}
\end{proposition}
\subsection{The orbit map} To compute the orbit map $\Psi$ and obtain Theorem \ref{thm:main}, we use that the following diagram commutes:
\[\xymatrix@C=0cm{
\{W\}\times \U(n+1)\ar[rr]^{\Psi}\ar[d] && \Gamma(\Cbun)\ar[d] \\
\{W\}\times \EU(n+1)\ar[dr]\ar[rr] && \Gamma(\Cbun)\times_{\U(n+1)} \EU(n+1) \ar[dl] \\
&\BU(n+1).
}\]
Therefore, by \cite{FHT} (bottom of page 204), we know that there is a rational model of this diagram that commutes on the nose: 
\[
    \xymatrix{
    \Lambda(s^{-1}W) && \Lambda(u\otimes 1,t\otimes 1,\ldots,t\otimes \beta_n) \ar[ll]_-{f}\\
    \Lambda(s^{-1}W)\otimes \Lambda W\ar[u]  && \Lambda(u\otimes 1,t\otimes 1,\ldots,t\otimes \beta_n)\otimes \Lambda W \ar[u]\ar[ll]_-{\bar{f}} \\
    &\Lambda W\ar[ur]\ar[ul]&
    }
\]
The map $f$ induced by $\Psi$ on rational models necessarily satisfies $\Psi(u\otimes 1) = 0$ by degree reasons. By the same reason, $\bar{f}(u\otimes 1)$ must be a multiple of $c_1$ (the only generator in degree $2$). Now, if $I(u\otimes 1)$ is the ideal generated by $u\otimes 1$, then the differential in Proposition \ref{prop:341} can be written as
\[
    \dd(t\otimes \beta_k) \equiv \left((1-d)^{n+1} - (1-d)^{k}\right)c_{n-k+1} \mod I(u\otimes 1)
\]
Therefore, if we write $I(s^{-1}c_1)$ for the ideal of $\Lambda(s^{-1}W)\otimes \Lambda W$ generated by $s^{-1}c_1$, then any choice that makes $\bar{f}$ a map of differential graded algebras is of the form 
\[
    \bar{f}(t\otimes \beta_k) \equiv \left((1-\degree)^{n+1} - (1-\degree)^{k}\right) s^{-1}c_{n-k+1} \mod I(s^{-1}c_1).
\]
As the diagram commutes on the nose, writing now $I(s^{-1}c_1)$ for the ideal generated by $c_1$ in $\Lambda (s^{-1}W)$, this gives
\[
    f(t\otimes \beta_k) \equiv \left((1-\degree)^{n+1} - (1-\degree)^{k}\right) s^{-1}c_{n-k+1} \mod I(s^{-1}c_1)
\]
as well. As a consequence, the generators from \eqref{eq:301} satisfy that
\[
    f(x_{2(n-k)+1}) \equiv \left((1-\degree)^{n+1} - (1-\degree)^{k}\right) s^{-1}c_{n-k+1} \mod I(s^{-1}c_1)
\]
Since we know that the orbit map factors through $PU(n+1)$, the terms in the ideal vanish. We then deduce the main theorem noticing that the coefficient $(1-\degree)^{n+1} - (1-\degree)^{k} $ vanishes only if $\degree = 0,1$ or $\degree = 2$ and $n-k+1$ is even. 

\appendix
\section{Almost-complex submanifolds and the scanning map}\label{ss:scanning}

In this section we define almost-complex submanifolds and their moduli space. We then introduce the scanning map, which may be understood as follows: There is a sheaf of topological spaces on $\CP^n$ that assigns to each open subset $U$ the space of proper submanifolds with a certain topology (introduced in \cite{GR-W}). This is not a homotopy sheaf, and the scanning map is the global sections of its homotopy sheafification. This viewpoint, not pursued here, is developed in \cite{R-WEmbedded}.

\subsection{Almost-complex submanifolds} Let $V$ be a complex vector space of dimension $n$. Let $\SSH{k}(V)$ be the space of pairs $(L,\gamma)$ where $L\in \Gr_{2k}(V)$ is a real $2k$-plane in $V$ and $\gamma\colon [0,1]\to \Gr_{2k}(V)$ is a path such that $\gamma(0) = L$ and $\gamma(1) \in \Gr_{k}^\bC(V)\subset \Gr_{2k}(V)$. Forgetting the datum of the path defines a fibre bundle
\[
    \SSH{k}(V)\lra \Gr_{2k}(V).
\]
This assignment is functorial on $V$ and thus may be applied fibrewise to the tangent bundle of any complex (or almost-complex) manifold $M$ to obtain a fibre bundle
\begin{equation}\label{eq:theta-structure}
    \SSH{k}(TM)\lra \Gr_{2k}(TM).
\end{equation}
Define the \emph{Gauss map} $\gauss\colon W\to \Gr_{2k}(TM)$ of a smooth submanifold $W\subset \CP^n$ of dimension $2k$ by sending each point of $W$ to its tangent space. 
\begin{definition}
    An \emph{almost-complex structure on a submanifold} $W\subset M$ is a lift of the Gauss map along the projection $\SSH{k}(TM)\to \Gr_{2k}(TM)$.
\end{definition}
Observe that an almost-complex structure on a submanifold $W$ makes $W$ an almost-complex manifold. Additionally, taking $M=\bR^{\infty}$, we have that \eqref{eq:theta-structure} becomes equivalent to $BU(k)\to BO(2k)$, that the Gauss map is the classifying map of the tangent bundle and that an almost-complex structure on a submanifold $W$ of $\bR^{\infty}$ is the same (up to homotopy) as an almost-complex structure on the manifold.

Let $W$ be a compact smooth manifold of real dimension $2k$, let $\Emb(W,M)$ be the space of smooth embeddings of $W$ into $M$ with the Whitney $C^{\infty}$ topology, and let $\ACEmb(W,M)$ be the space of pairs $(f,\ell)$ where $f$ is an embedding and $\ell$ is a lift of the composition $\gauss\circ f$ along \eqref{eq:theta-structure}. Let $\Diff(W)$ be the group of diffeomorphisms of $W$, which acts on these spaces by precomposition. Write $\cA_k(M)$ for the set of almost-complex submanifolds of $M$ of dimension $2k$. There is a bijection 
\[
    \cA_k(M)\cong \coprod_{[W]}\ACEmb(W,M)/\Diff(W)
\]
where $W$ runs along diffeomorphism classes of compact manifolds of dimension $2k$, and we use it to topologise the set $\cA_k(M)$.

\subsection{The scanning map}
Let $V$ be a complex vector space of dimension $n$ with a hermitian metric. Let $\gamma_{k}^\perp(V)\to \Gr_{2k}(V)$ be the orthogonal complement of the tautological bundle of the Grassmannian of real $2k$-planes in $V$. Consider the pullback 
\[
    \xymatrix{
    \tilde{\gamma}_{k}^\perp(V)\ar@{-->}[r]\ar@{-->}[d] & \gamma_{2k}^{\perp}(V)\ar[d] \\
    \SSH{k}(V)\ar[r]& \Gr_{2k}(V)
    }
\]
Define a functor $\SS{k}\colon \Vect(\bC)\to \Top_*$ from the category of hermitian complex vector spaces to the category of based topological spaces by sending a complex vector space $V$ to the Thom space of the vector bundle $\tilde{\gamma}_{k}^\perp(V)\to \widetilde{\Gr}_{k}(V)$. We may apply this construction to the tangent bundle of an almost-complex manifold $M$ obtaining a fibre bundle
\begin{align*}
    \SS{n-k}(TM)&\lra M 
\end{align*}
with a canonical section.

Following the construction of \cite[p.~1388]{CR-W}, there is a map (called \emph{scanning map}) 
\[
\scan\colon \cA_k(M) \lra \Gamma_c(\SS{k}(TM))
\]
to the space of compactly supported sections of $\SS{k}(TM)\to M$. This map is equivariant with respect to the action of $\Diff_c(M)$ by postcomposition on the left-hand side and to the action by conjugation on the right-hand side.

\subsection{Almost-complex projective hypersurfaces} Let us particularise in the case $M=\CP^n$ and $k=n-1$. In this case, the fibre bundle $\SSH{n-1}(T\CP^n)$ is homotopy equivalent to the fibrewise Thom space of the fibrewise tautological bundle $\gamma_1(T\CP^n)\to \Gr_1^{\bC}(\CP^n)$, which we denote as
\[
    \Cbun\to\CP^n.
\]
This bundle is also homeomorphic to the bundle
\[
    \PCbun\to \CP^n.
\]
The space $U_{d,n}$ sits inside $\CS$ as the subspace of almost-complex hypersurfaces with constant deformation. The composition
\[
    U_{d,n}\lra \CS\lra \Gamma(\PCbun)
\]
is the scanning map studied in this paper.

There is a class $u\in H^2(\Cbun;\bQ)$ whose restriction to each fibre is the Thom class of the corresponding Thom space. If $f$ is a section of this bundle, evaluating $f^{*}(u)$ on the fundamental class of $\CP^1\subset \CP^n$ defines a map
\[
    \degree\colon \GCbun\lra H^2(\CP^n) \cong \bZ
\]
\begin{proposition}
    The map $\degree$ is a bijection on connected components. Additionally, the section that arises from scanning an almost-complex hypersurface of degree $\degree$ is in the component indexed by $\degree$.
\end{proposition}
\begin{proof}
    As observed before, the fibre of the bundle $\PP{1}(T\CP^n)\to \CP^n$ is homeomorphic to $\CP^n$, and the Thom class restricts to the standard generator of $H^2(\CP^n)$. Therefore, the obstructions to the existence of a section live in the trivial groups $H^k(\CP^n;\pi_{k-1}(\CP^n))$. On the other hand, there is a single obstruction to find a homotopy between any given two sections, which lives in $H^2(\CP^n;\pi_2 (\CP^n))\cong \bZ$. The Hurewicz homomorphism $\pi_2(\CP^n)\to H_2(\CP^n)$ is an isomorphism and therefore we have isomorphisms
    \[
        H^2(\CP^n;\pi_2 (\CP^n)) \cong H^2(\CP^n;H_2(\CP^n)) \cong \hom(H_2(\CP^n),H_2(\CP ^n))
    \]
    Additionally, since the class $u\in H^2(\PP{1}(T\CP^n))$ restricts to the Thom class of the fibre, by the Leray-Hirsch theorem
    \[
        H_2(\PP{1}(T\CP^n))\cong H_2(\CP^n)\oplus H_2(\CP^n).
    \]
    A section of the bundle determines therefore a homomorphism
    \[
        \hom(H_2(\CP^n),H_2(\CP^n))
    \]
    which under the previous chain of isomorphisms corresponds to the obstruction. The evaluation at $u$ of the dual of this homomorphism is precisely the map $\degree$. 

    For the second statement, notice that the class $u\in H^2(\PP{1}(T\CP^n))$ is the Poincaré dual of the submanifold $\SSH{1}(T\CP^n)$. If $W\in \CS$, then $\scan(W)$ is a section transverse to this submanifold, and therefore its inverse image $W$ is Poincaré dual to $\scan(W)^*(u)$. Its evaluation at the fundamental class of $\CP^1$ computes the degree $\degree$ of $W$.
 \end{proof}



\bibliographystyle{alpha}
\bibliography{biblio}
\end{document}